\documentclass[10pt,twoside,openany,UTF8,CJK]{article}
\usepackage{amssymb,amsthm}
\usepackage[tbtags]{amsmath}
\usepackage{cite}
\usepackage{indentfirst}
\usepackage{cases}
\usepackage{graphicx}
\usepackage{subfigure}
\usepackage{lineno}
\usepackage{enumerate}
\usepackage{threeparttable}
\usepackage{multirow}
\usepackage{booktabs}
\usepackage{upgreek}  
\usepackage{bm}
\usepackage{CJK}
\usepackage{fancyhdr}
\usepackage{ifthen}
\usepackage{lipsum}
\allowdisplaybreaks[4]
\usepackage[misc]{ifsym}
\usepackage{caption}
\usepackage{float}
\usepackage{appendix}
\usepackage[format=hang]{caption}
\usepackage{algorithm,algorithmic}

\usepackage{charter}  
\usepackage{indentfirst}  
\RequirePackage{bibspacing}  
\usepackage{soul}

\newcommand{\Rmnum}[1]{\uppercase\expandafter{\romannumeral #1}} 

\usepackage[colorlinks,
linkcolor=blue,
anchorcolor=blue,
citecolor=red]{hyperref}

\newtheorem{Lemma}{Lemma}[section]
\newtheorem{Theorem}{Theorem}[section]
\newtheorem{Remark}{Remark}[section]

\newcounter{saveeqn}

\makeatletter
\evensidemargin
\oddsidemargin
\makeatother

\title{\bf  On the conservation properties of the two-level linearized methods for Navier-Stokes equations}
\author{
	Xi Li,\footnote{School of Mathematics, Sichuan University, Chengdu, Sichuan 610064, China (li\_xi@stu.scu.edu.cn). The work of this author was supported by the National Natural Science Foundation of China(Grant No. 11971337).}
	\ and Minfu Feng\footnote{Corresponding author. School of Mathematics, Sichuan University, Chengdu, Sichuan 610064, China (fmf@scu.edu.cn). The work of this author was supported by the National Natural Science Foundation of China(Grant No. 11971337).}
}
\date{}

\begin{document}
	\maketitle
	\newcommand\blfootnote[1]{%
		\begingroup
		\renewcommand\thefootnote{}\footnote{#1}%
		\addtocounter{footnote}{-1}%
		\par\setlength\parindent{2em}
		\endgroup
	}
	\captionsetup[figure]{labelfont={bf},labelformat={default},labelsep=period,name={Fig.}}
	\captionsetup[table]{labelfont={bf},labelformat={default},labelsep=period,name={Tab.}}
	
	\begin{abstract}
		This manuscript is devoted to investigating the conservation laws of incompressible Navier-Stokes equations(NSEs), written in the energy-momentum-angular momentum conserving(EMAC) formulation, after being linearized by the two-level methods. With appropriate correction steps(e.g., Stoke/Newton corrections), we show that the two-level methods, discretized from EMAC NSEs, could preserve momentum, angular momentum, and asymptotically preserve energy. Error estimates and (asymptotic) conservative properties are analyzed and obtained, and numerical experiments are conducted to validate the theoretical results, mainly confirming that the two-level linearized methods indeed possess the property of (almost) retainability on conservation laws. Moreover, experimental error estimates and optimal convergence rates of two newly defined types of pressure approximation in EMAC NSEs are also obtained.   \\
		
		\noindent {\bf Keywords: }{Conservative methods; EMAC formulation; structure-preserved discretization.}\\
	\end{abstract}
	
	\baselineskip 15pt
	\parskip 10pt
	\setcounter{page}{1}
	\vspace{-0.5cm}
	\section{Introduction}
	\vspace{-0.5cm}
	We consider the following nonstationary incompressible Navier-Stokes equations(NSEs) in primitive variables with Dirichlet boundary conditions: 
	\begin{equation}\label{ContEqua}
		\setlength{\abovedisplayskip}{0.2cm}
		\begin{aligned}
			&\partial_t \boldsymbol{u} + \boldsymbol{u}\cdot\nabla\boldsymbol{u} - \nu\Delta \boldsymbol{u} + \nabla p = \boldsymbol{f},\quad\text{and}\quad\nabla\cdot \boldsymbol{u} = 0,\quad \text{in}\;\Omega\times(0,T],\\
			&\boldsymbol{u} = \boldsymbol{u}_D,\;\text{on}\;\partial\Omega,\quad\boldsymbol{u}(0,x) = \boldsymbol{u}_{0}(x),\; \text{in}\;\partial\Omega,
		\end{aligned}
	    \setlength{\belowdisplayskip}{0.2cm}
	\end{equation} 
	where $\Omega \subset \mathbb{R}^2$ be an convex polygonal domain with boundary $\partial \Omega$, and $(0,T]$ denotes time interval. It is well-known that the smooth solution of (\ref{ContEqua}) obeys a series of conservative laws, including the conservation of energy, momentum, etc., and pursuing their discrete counterparts relating to the numerical solutions, in computational fluid dynamics community, is a long-standing goal since a better adherence to those laws would lead to more physically accurate numerical solutions(e.g., see \cite{Helicity-ARFM-1992,LongerAccu-2020-CMAME}), as well as a challenge, especially if choosing continuous Galerkin(CG) discretization, the weakly enforcing of the solenoidal constraint would typically aggravate the violation of those laws. Apart from the conservation laws, the nonlinearity in NSEs also poses computational bottlenecks, making the efficient linearized methods appealing or vital in some practical or engineering applications. Given those mentioned two intrinsic natures of (\ref{ContEqua}), an efficient and conservative numerical scheme for NSEs becomes attractive and deserves more research.    \\
	\indent However, the general (linearized) numerical schemes for (\ref{ContEqua}) would damage those conservation properties possessed by continuous NSEs; for example, spatial and CG discretization, or efficiently IMEX discretization for temporal variable would alter those balances\cite{EMAC-2017-JCP,Linearized-EMAC-ANM-2019}. Therefore, much attention needs to be paid if one wants to efficiently obtain conservative numerical solutions, especially appropriate CG formulations and linearized methods are required. To preserve as many physical quantities as possible, Charnyi et al.\cite{EMAC-2017-JCP} re-writes equivalently continuous NSEs in the EMAC formulation(Energy, Momentum, and Angular momentum (EMA)-Conserving), thereby obtaining a nonlinear EMAC scheme. However, they also have pointed out that the direct IMEX linearization would damage those conservation properties\cite{Linearized-EMAC-ANM-2019}.  \\
	\indent The main purpose of this article is to design the (asymptotically) conservative and efficient CG numerical schemes for (\ref{ContEqua}). Specifically, for the NSEs in EMAC variation, we propose two (asymptotically) conservative two-level linearized continuous Galerkin discrete schemes. Convergence results about newly proposed schemes are presented, as well as the (asymptotic) conservation analysis and results for energy, momentum, and angular momentum. The main contributions are two-fold. On the one hand, we propose two preserved and linearized schemes for nonlinear EMAC inertial term; on the other hand, our schemes can also be viewed as an improvement of the classical two-level methods\cite{TwoLevel-1994-SISC} on conservative properties, which, to our knowledge, are seldom considered before. Numerical experiments validate the two-level schemes' convergence and (asymptotic) conservation properties. In addition, we give two ways of calculating errors and optimal convergence rates of pressure in EMAC NSEs, which, to the best of our knowledge, is the first time to consider the (experimental) error and convergence analysis on the pressure in the EMAC form.  \\
	\indent This draft proceeds as follows. In Section \ref{section-2}, some notation and preliminaries will be given. Section \ref{section-3} presents two-level schemes, along with error and conservation analysis. In Section \ref{section-4}, we utilize numerical experiments to validate our schemes. Throughout this draft, we use $C$ to denote a positive constant independent of $\Delta t$, $h$, not necessarily the same at each occurrence. $a \lesssim b$ means $a \leq Cb$. $L^p(\Omega)$, $W^{m,p}(\Omega)$ and $H^m(\Omega)$, $1\leq p\leq \infty$, and their inner-products and norms are defined in the standard ways\cite{Brenner2008}. Vector analogs of the Sobolev spaces and vector-valued functions are denoted by boldface letters.
	\vspace{-0.5cm}
	\section{Notation and preliminaries}\label{section-2}
	\vspace{-0.5cm}
	We consider the natural velocity and pressure spaces for NSEs, respectively, by 
	$$
	\setlength{\abovedisplayskip}{0.2cm}
	\boldsymbol{V} := \boldsymbol{H}^1_0(\Omega) = \left\{\boldsymbol{v} \in \boldsymbol{H}^1(\Omega): \;\boldsymbol{v}|_{\partial \Omega}=0\right\}, \quad
	Q := L^2_0(\Omega) = \left\{q \in L^2(\Omega): \;\int_{\Omega} q \;{\rm d} x=0\right\},
	\setlength{\belowdisplayskip}{0.2cm}
	$$
	and denote $\boldsymbol{Y} := \boldsymbol{L}^2(\Omega)$. The divergence-free subspace of $\boldsymbol{V}$ and of $\boldsymbol{Y}$ are
	$$
	\setlength{\abovedisplayskip}{0.2cm}
	\boldsymbol{V}_{\text{div}} := \left\{\boldsymbol{v} \in \boldsymbol{V}: \;\nabla\cdot\boldsymbol{v} = 0\right\}, \quad \boldsymbol{Y}_{\text{div}} := \left\{\boldsymbol{v} \in \boldsymbol{Y}: \;\nabla\cdot\boldsymbol{v} = 0, \;  \boldsymbol{v}\cdot\boldsymbol{n}|_{\partial\Omega} = 0\right\}.
	\setlength{\belowdisplayskip}{0.2cm}
	$$
	\indent Define the inertial term in convective form(Conv) $b_{conv}:\boldsymbol{V}\times\boldsymbol{V}\times\boldsymbol{V} \rightarrow \mathbb{R}$, or EMAC form(EMAC) $b_{emac}:\boldsymbol{V}\times\boldsymbol{V}\times\boldsymbol{V} \rightarrow \mathbb{R}$ as
	$$
	\setlength{\abovedisplayskip}{0.2cm}
	b_{conv}(\boldsymbol{u},\boldsymbol{v},\boldsymbol{w}) := \left((\boldsymbol{u} \cdot \nabla )\boldsymbol{v}, \boldsymbol{w}\right), \quad 
	b_{emac}(\boldsymbol{u},\boldsymbol{v},\boldsymbol{w}) := \left(2\textbf{D}(\boldsymbol{u})\boldsymbol{v}, \boldsymbol{w}\right) + \left((\nabla \cdot \boldsymbol{u})\boldsymbol{v}, \boldsymbol{w}\right),
	\setlength{\belowdisplayskip}{0.2cm}
	$$
	where $\textbf{D}(\boldsymbol{u}) := (\nabla \boldsymbol{u} + (\nabla \boldsymbol{u})^\text{T})/2$ denotes the symmetric part of $\nabla \boldsymbol{u}$. With simple calculations, we can get
	\begin{subequations}
		\setlength{\abovedisplayskip}{0.2cm}
		\begin{align}
			b_{conv}(\boldsymbol{u},\boldsymbol{v},\boldsymbol{w}) &= -b_{conv}(\boldsymbol{u},\boldsymbol{w},\boldsymbol{v}) - \left(\left(\nabla\cdot \boldsymbol{u}\right)\boldsymbol{v}, \boldsymbol{w}\right),\quad \text{for }\boldsymbol{u} \in \boldsymbol{V}, \boldsymbol{v}\text{ and }\boldsymbol{w} \in \boldsymbol{H}^1, \label{PartInte-CONV}\\
			b_{conv}(\boldsymbol{u},\boldsymbol{w},\boldsymbol{w}) &= -\frac12\left(\left(\nabla\cdot \boldsymbol{u}\right)\boldsymbol{w}, \boldsymbol{w}\right),\quad \text{for }\boldsymbol{u} \in \boldsymbol{V}, \boldsymbol{v}\text{ and }\boldsymbol{w} \in \boldsymbol{H}^1. \label{PartInte-2-CONV}
		\end{align}
	    \setlength{\belowdisplayskip}{0.2cm}
	\end{subequations}
    Thus, we can express $b_{emac}$ in terms of $b_{conv}$ to have 
    \begin{subequations}
    	\setlength{\abovedisplayskip}{0.2cm}
    	\begin{align}
    		b_{emac}(\boldsymbol{u},\boldsymbol{v},\boldsymbol{w}) &= b_{conv}(\boldsymbol{v},\boldsymbol{u},\boldsymbol{w}) + b_{conv}(\boldsymbol{w},\boldsymbol{u},\boldsymbol{v}) + \left((\nabla \cdot \boldsymbol{u})\boldsymbol{v}, \boldsymbol{w}\right), \label{Rela-2-EMAC-CONV}  \\
    		b_{emac}(\boldsymbol{u},\boldsymbol{v},\boldsymbol{w}) &= b_{conv}(\boldsymbol{v},\boldsymbol{u},\boldsymbol{w}) + b_{conv}(\boldsymbol{w},\boldsymbol{u},\boldsymbol{v}) - b_{conv}(\boldsymbol{u},\boldsymbol{v},\boldsymbol{w}) - b_{conv}(\boldsymbol{u},\boldsymbol{w},\boldsymbol{v}),  \label{Rela-EMAC-CONV}
    	\end{align}
    	\setlength{\belowdisplayskip}{0.2cm}
    \end{subequations}
    and (e.g., see \cite[(10)]{Proj-Shen-1992-NM} and \cite[Lemma 2.2]{TwoLevel-NSE-Layton-Tabiska-1998-SINUM})
    \begin{equation}\label{IneqBconv}
    	\setlength{\abovedisplayskip}{0.2cm}
    	b_{conv}(\boldsymbol{u}, \boldsymbol{v}, \boldsymbol{w}) \lesssim
    	\left\{\begin{array}{l}
    		\|\boldsymbol{u}\|_1\|\boldsymbol{v}\|_1\|\boldsymbol{w}\|_1, \\
    		\|\boldsymbol{u}\|_2\|\boldsymbol{v}\|_0\|\boldsymbol{w}\|_1, \\
    		\|\boldsymbol{u}\|_1\|\boldsymbol{v}\|_0\|\boldsymbol{w}\|_2, \\
    		\|\boldsymbol{u}\|_0\|\boldsymbol{v}\|_2\|\boldsymbol{w}\|_1,  \\
    		\|\boldsymbol{u}\|_0\|\boldsymbol{v}\|_1\|\boldsymbol{w}\|_2.
    	\end{array}
        \right.
        \setlength{\belowdisplayskip}{0.2cm}
    \end{equation}
    Therefore, based on (\ref{IneqBconv}) and (\ref{Rela-2-EMAC-CONV})-(\ref{Rela-EMAC-CONV}), we can obtain the similar estimates about $b_{emac}((\boldsymbol{u},\boldsymbol{v},\boldsymbol{w}))$ with $b_{conv}((\boldsymbol{u},\boldsymbol{v},\boldsymbol{w}))$ whose proof are omitted for simplicity.
    \begin{Lemma}\label{Lemma-IneqBemac}
    	The convection term in EMAC form defined as $b_{emac}(\boldsymbol{u},\boldsymbol{v},\boldsymbol{w}) := \left(2\textbf{D}(\boldsymbol{u})\boldsymbol{v}, \boldsymbol{w}\right) + \left((\nabla \cdot \boldsymbol{u})\boldsymbol{v}, \boldsymbol{w}\right)$ satisfies the following estimates:
    	\begin{equation}\label{IneqBemac}
    		\setlength{\abovedisplayskip}{0.2cm}
    		b_{emac}(\boldsymbol{u}, \boldsymbol{v}, \boldsymbol{w}) 
    		\lesssim \min\{\|\boldsymbol{u}\|_1\|\boldsymbol{v}\|_1\|\boldsymbol{w}\|_1, \|\boldsymbol{u}\|_0\|\boldsymbol{v}\|_2\|\boldsymbol{w}\|_1\}.
    		\setlength{\belowdisplayskip}{0.2cm}
    	\end{equation}
    \end{Lemma}
    \vspace{-0.5cm}
    Based on the preparations above, the continuous variational EMAC form of the Navier-Stokes equations (\ref{ContEqua}) is to find $(\boldsymbol{u}, p) \in \left(L^{\infty}(0,T; \boldsymbol{Y}_{\text{div}}) \cap L^2(0,T; \boldsymbol{V}) \right) \times L^2(0,T; Q)$ such that
    \begin{equation}\label{ContVariEqua}
    	\setlength{\abovedisplayskip}{0.2cm}
    	\begin{aligned}
    		\left(\partial_t \boldsymbol{u},\boldsymbol{v}\right) + \nu \left(\nabla \boldsymbol{u},\nabla \boldsymbol{v}\right) + b_{emac}(\boldsymbol{u},\boldsymbol{u},\boldsymbol{v}) - \left(\nabla \cdot \boldsymbol{v}, p\right) & =\left(\boldsymbol{f}, \boldsymbol{v}\right), \\
    		\left(\nabla \cdot \boldsymbol{u}, q\right) & =0.
    	\end{aligned}
        \setlength{\belowdisplayskip}{0.2cm}
    \end{equation}
    \vspace{-1.0cm}
	\subsection{Finite element spaces}
	\vspace{-0.5cm}
	Let ${\mathcal{T}_d}$ be a uniformly regular family of triangulation of $\overline{\Omega}$ where $d := \max_{K\in \mathcal{T}_d}\{d_K | d_K := \text{diam}(K)\}$(here $d=H,h$ with $h < H$). The fine mesh ${\mathcal{T}_h}$, for simplicity, can be thought of as refined from the coarse mesh ${\mathcal{T}_H}$ by some mesh refinement methods. The conforming finite element spaces $(\boldsymbol{V}_d, Q_d) \subset (\boldsymbol{V}, Q)$($d=H,h$) are considered and therefore there holds $(\boldsymbol{V}_H, Q_H) \subset (\boldsymbol{V}_h, Q_h)$. We remark that this inclusion condition is not necessary for two-level methods, and it is assumed here only for analytical brevity. Moreover, these two space-pairs are assumed to satisfy the following compatibility conditions and approximation properties (where $d=H, h$): for some $\beta$ independent of $d(H$ or $h)$, there exists
	\begin{equation}\label{InfSupCond}
		\setlength{\abovedisplayskip}{0.2cm}
		\inf\limits_{q_d \in Q_d}\sup\limits_{\boldsymbol{v}_d \in \boldsymbol{V}_d} \frac{(\nabla\cdot \boldsymbol{v}_d, q_d)}{\|q_d\|_0\|\nabla \boldsymbol{v}_d\|_0} \geq \beta > 0, 
		\setlength{\belowdisplayskip}{0.2cm}
    \end{equation}
	and for $\forall \boldsymbol{v} \in \boldsymbol{H}^2(\Omega) \cap \boldsymbol{V}$, $q \in H^1(\Omega) \cap Q$; and there exist approximations $\Pi_d \boldsymbol{v} \in \boldsymbol{V}_d$, $\Pi_d q \in Q_d$ such that
	\begin{equation}\label{ApproProper_FEspace}
		\setlength{\abovedisplayskip}{0.2cm}
		\|\boldsymbol{v} - \Pi_d \boldsymbol{v}\|_1 \lesssim d\|\boldsymbol{v}\|_2, \quad \|q - \Pi_d q\|_0 \lesssim d\|q\|_1,
		\setlength{\belowdisplayskip}{0.2cm}
	\end{equation}
    where $\Pi_d:L^2 \rightarrow V_d$ is the classical $L^2$-projection defined as
    $$
    \setlength{\abovedisplayskip}{0.2cm}
    (v - \Pi_dv, w_d) = 0, \quad\forall w_d \in V_d.
    \setlength{\belowdisplayskip}{0.2cm}
    $$
    The inverse inequality holds, which is implied by the shape-regular of the mesh
    \begin{equation}\label{InveIneq}
    	\setlength{\abovedisplayskip}{0.2cm}
    	\|\boldsymbol{v}_d\|_1 \lesssim d^{-1} \|\boldsymbol{v}_d\|_0, \quad \forall \boldsymbol{v}_d \in \boldsymbol{V}_d,
    	\setlength{\belowdisplayskip}{0.2cm}
    \end{equation}
    and Poincaré-Friedrichs inequality can be obtained by homogeneity of its Sobolev space
    \begin{equation}\label{PF-ineq}
    	\setlength{\abovedisplayskip}{0.2cm}
    	\|\boldsymbol{v}\|_0 \lesssim\|\boldsymbol{v}\|_1, \quad \forall \boldsymbol{v} \in \boldsymbol{V}.
    	\setlength{\belowdisplayskip}{0.2cm}
    \end{equation}
    \vspace{-0.8cm}
	\section{Numerical scheme and analysis}\label{section-3}
	\vspace{-0.5cm}
    \subsection{One-Level EMAC schemes}
    \vspace{-0.5cm}
    We firstly discrete the continuous variational form (\ref{ContVariEqua}) on one-single mesh ${\mathcal{T}_h}$. For the sake of explanation, we adopt backward Euler finite difference(BDF1) to discretize temporal variable, and inf-sup stable mixed finite element pair $(\boldsymbol{V}_h, Q_h)$ in space. For the given time $T$, we adopt a uniform partition of the time interval $[0,T]$ with some fixed time-step $\Delta t$, i.e., $t_n:=n\Delta t$ where $0 \leq n \leq N :=\lfloor T/\Delta t\rfloor$ where the symbol $\lfloor\cdot\rfloor$ denotes the rounding-down operation. The one-level numerical scheme about EMAC form (\ref{ContVariEqua}) states as: Given $\boldsymbol{u}^0_h = \Pi_h\boldsymbol{u}_0$, for $0 \leq n \leq N-1$, finding $(\boldsymbol{u}^{n+1}_h,p^{n+1}_h) \in \boldsymbol{V}_h \times Q_h$, such that
    \begin{equation}\label{OneLevel-EMAC-Nonlinear}
    	\begin{aligned}
    		\left(\frac{\boldsymbol{u}^{n+1}_h - \boldsymbol{u}^{n}_h}{\Delta t}, \boldsymbol{v}_h\right) &+ \nu (\nabla\boldsymbol{u}^{n+1}_h,\nabla\boldsymbol{v}_h) + b_{emac}(\boldsymbol{u}^{n+1}_h,\boldsymbol{u}^{n+1}_h,\boldsymbol{v}_h) \\
    		& + (\nabla\cdot \boldsymbol{v}_h,p^{n+1}_h) - (\nabla\cdot \boldsymbol{u}^{n+1}_h,q_h) = (\boldsymbol{f}^{n+1},\boldsymbol{v}_h).
    	\end{aligned}
    \end{equation}
    \vspace{-0.5cm}
	\subsection{Two-Level EMAC schemes}
	\vspace{-0.5cm}
	We here introduce the two-level linearized scheme of (\ref{ContVariEqua}) based on Stokes and Newton corrections, and we call it EMAC-TwoLevel(Stokes) and EMAC-TwoLevel(Newton) respectively. Given time $T$, let $\boldsymbol{u}^0_h = \Pi_h \boldsymbol{u}_0$, $\boldsymbol{u}^0_H = \Pi_H \boldsymbol{u}_0$, then for some positive integer $n$($0 \leq n \leq N$), the EMAC-TwoLevel(Stokes) is achieved by the following two steps:  \\
    \noindent\textbf{\text{Step \Rmnum{1}: }}
    \begin{minipage}[t]{0.9\linewidth} 
    Find a coarse-mesh solution $(\boldsymbol{u}^{n+1}_H, p^{n+1}_H) \in (\boldsymbol{X}_H, Q_H)$ by solving the nonlinear equations (\ref{OneLevel-EMAC-Nonlinear}).
    \end{minipage} 
    \vspace{0.0cm} \\
    \noindent\textbf{\text{Step \Rmnum{2}:}}	
    \begin{minipage}[t]{0.9\linewidth}
    Find a fine-mesh solution $(\boldsymbol{u}^{n+1}_h, p^{n+1}_h) \in (\boldsymbol{X}_h, Q_h)$ by solving the following linearized Stokes equations  
    \begin{equation}\label{TwoLevel-EMAC-Stokes-Step2}
    \begin{aligned}
    	\left(\frac{\boldsymbol{u}^{n+1}_h - \boldsymbol{u}^{n}_h}{\Delta t}, \boldsymbol{v}_h\right) 
    	& + \nu (\nabla\boldsymbol{u}^{n+1}_h,\nabla\boldsymbol{v}_h) + b_{emac}(\boldsymbol{u}^{n+1}_H,\boldsymbol{u}^{n+1}_H,\boldsymbol{v}_h)   \\
    	& + (\nabla\cdot \boldsymbol{v}_h,p^{n+1}_h) - (\nabla\cdot \boldsymbol{u}^{n+1}_h,q_h) = (\boldsymbol{f}^{n+1},\boldsymbol{v}_h).
    \end{aligned}
    \end{equation}	
    \end{minipage}	
    	
	\indent The EMAC-TwoLevel(Newton) is given by:  \\
	\noindent\textbf{\text{Step \Rmnum{1}: }}
	\begin{minipage}[t]{0.9\linewidth} 
		Find a coarse-level solution $(\boldsymbol{u}^{n+1}_H, p^{n+1}_H) \in (\boldsymbol{X}_H, Q_H)$ by solving the nonlinear equations (\ref{OneLevel-EMAC-Nonlinear}).
	\end{minipage} 
	\vspace{0.0cm} \\
	\noindent\textbf{\text{Step \Rmnum{2}:}}	
	\begin{minipage}[t]{0.9\linewidth}
		Find a fine-level solution $(\boldsymbol{u}^{n+1}_h, p^{n+1}_h) \in (\boldsymbol{X}_h, Q_h)$ by solving the following linearized Stokes equations  
		\begin{equation}\label{TwoLevel-EMAC-Newton-Step2}
			\begin{aligned}
				\!\!\!\!\!\!\!\!\left(\frac{\boldsymbol{u}^{n+1}_h - \boldsymbol{u}^{n}_h}{\Delta t}, \boldsymbol{v}_h\right) 
				& + \nu (\nabla\boldsymbol{u}^{n+1}_h,\nabla\boldsymbol{v}_h) + b_{emac}(\boldsymbol{u}^{n+1}_H,\boldsymbol{u}^{n+1}_h,\boldsymbol{v}_h) + b_{emac}(\boldsymbol{u}^{n+1}_h,\boldsymbol{u}^{n+1}_H,\boldsymbol{v}_h)   \\
				& + (\nabla\cdot\boldsymbol{v}_h ,p^{n+1}_h) - (\nabla\cdot \boldsymbol{u}^{n+1}_h,q_h) = (\boldsymbol{f}^{n+1},\boldsymbol{v}_h) + b_{emac}(\boldsymbol{u}^{n+1}_H,\boldsymbol{u}^{n+1}_H,\boldsymbol{v}_h).
			\end{aligned}
		\end{equation}	
	\end{minipage}	
    \vspace{-0.5cm}
	\subsection{Error estimate}
	\vspace{-0.5cm}
	In this subsection, we will derive the convergence results of TwoLevel(Stokes) EMAC scheme (\ref{OneLevel-EMAC-Nonlinear}), (\ref{TwoLevel-EMAC-Stokes-Step2}) and TwoLevel(Newton) EMAC scheme (\ref{OneLevel-EMAC-Nonlinear}), (\ref{TwoLevel-EMAC-Newton-Step2}) to validate \emph{a prior} these two numerical schemes' convergence(as $h,H \rightarrow 0$). To this end, we first give the error estimate for the One-Level EMAC scheme (\ref{OneLevel-EMAC-Nonlinear}), and for the proof of this result, we have noticed that Lemma \ref{Lemma-IneqBemac} shows the similar upper bounds of EMAC form as its counterpart of Conv form, which suggests that similar estimate techniques can be used in the error estimate of One-Level EMAC scheme (\ref{OneLevel-EMAC-Nonlinear}) and thereby obtaining the similar results; therefore, we omit those proofs owing to space constraints which will be found in \cite[Theorem 3.1]{Heywood-Rannacher-1982-SINUM-1} after complementing the analysis of temporal truncation error without essential difficulties.
	\begin{Theorem}[Convergence for OneLevel EMAC]\label{Theorem-OneLevel-EMAC-Conv}
		Under the conditions (\ref{InfSupCond})-(\ref{PF-ineq}), the One-Level EMAC scheme (\ref{OneLevel-EMAC-Nonlinear}) admits a unique solution $\boldsymbol{u}^{n}_H$ for $\forall n$. Moreover, it holds that
		\begin{equation}
			\setlength{\abovedisplayskip}{0.2cm}
			\|\boldsymbol{u}^{n} - \boldsymbol{u}^{n}_H\|_0 + H\|\nabla(\boldsymbol{u}^{n} - \boldsymbol{u}^{n}_H)\|_0 \lesssim H^2 + \Delta t.
		    \setlength{\belowdisplayskip}{0.2cm}
    	\end{equation}
	\end{Theorem}
    \vspace{-0.5cm}
    Here are the results of convergence about TwoLevel(Stokes) EMAC scheme (\ref{OneLevel-EMAC-Nonlinear}),(\ref{TwoLevel-EMAC-Stokes-Step2}) and TwoLevel (Newton) EMAC scheme (\ref{OneLevel-EMAC-Nonlinear}), (\ref{TwoLevel-EMAC-Newton-Step2}) without proof again, since those are highly similar, after substituting (\ref{IneqBconv}) with (\ref{IneqBemac}), with the two-level scheme with Stokes correction on fine mesh in Conv form in \cite[Theorem 6.6]{TwoLevel-UNSE-JCM-2004-2}, and two-level scheme with Newton correction on fine mesh in Conv form in \cite[Theorem 4.1]{TwoLevel-ThreeStep-UNSE-2017-JNM} (with $\delta = 1$). Detailed proof can be obtained in those mentioned references.
	\begin{Theorem}[Convergence for TwoLevel(Stokes) EMAC]\label{Theorem-Conv-Stokes}
		Under the conditions (\ref{InfSupCond})-(\ref{PF-ineq}), the solution $\boldsymbol{u}^{n}_h$ of TwoLevel(Stokes) EMAC scheme (\ref{OneLevel-EMAC-Nonlinear}),(\ref{TwoLevel-EMAC-Stokes-Step2}) and $\boldsymbol{u}^{n}$ in (\ref{ContVariEqua}) at $t=t_{n}$, for $\forall n$, satisfy
		\begin{equation}
			\setlength{\abovedisplayskip}{0.2cm}
			\|\nabla(\boldsymbol{u}^{n} - \boldsymbol{u}^{n}_h)\|_0 \lesssim h + H^2 + \Delta t.
			\setlength{\belowdisplayskip}{0.2cm}
		\end{equation}
	\end{Theorem}
    \begin{Theorem}[Convergence for TwoLevel(Newton) EMAC]\label{Theorem-Conv-Newton}
    	Under the conditions (\ref{InfSupCond})-(\ref{PF-ineq}), the solution $\boldsymbol{u}^{n}_h$ of TwoLevel(Newton) EMAC scheme (\ref{OneLevel-EMAC-Nonlinear}),(\ref{TwoLevel-EMAC-Newton-Step2}) and $\boldsymbol{u}^{n}$ in (\ref{ContVariEqua}) at $t=t_{n}$, for $\forall n$, satisfy
    	\begin{equation}
    		\setlength{\abovedisplayskip}{0.2cm}
    		\|\nabla(\boldsymbol{u}^{n} - \boldsymbol{u}^{n}_h)\|_0 \lesssim h + H^2 + \Delta t.
    		\setlength{\belowdisplayskip}{0.2cm}
    	\end{equation}
    \end{Theorem}
    \vspace{-0.2cm}
	\begin{Remark}
		We remark on the error estimate of pressure. We omit the error estimate about pressure since the primary purpose of this manuscript is to study, analyze, and verify some conservation properties related to velocity. We can also, for completeness, analyze pressure's convergence. To this end, we can define two types of pressure error estimates: the first one is fixing continuous, primal pressure $p(\boldsymbol{x},t)$ and recovering its discrete counterpart $p^n_{h,primal}$ a posterior as $p^n_{h,primal} := p^n_{h} + \frac12|\boldsymbol{u}_h^n|^2$ after solving (\ref{TwoLevel-EMAC-Stokes-Step2}) or (\ref{TwoLevel-EMAC-Newton-Step2}); the second type is redefining a prior continuous EMAC pressure $p_{emac}(\boldsymbol{x},t) := p(\boldsymbol{x},t) - \frac12|\boldsymbol{u}(\boldsymbol{x},t)|^2 + \lambda$ and $\lambda := \int_{\Omega}\frac12|\boldsymbol{u}|^2\text{d}x$, thus obtaining error estimate $\|p_{emac} - p_h^n\|_{L^2}$. We present the errors and optimal convergence rates in the forthcoming numerical experiment and omit their theoretical analysis which will be presented in subsequent work. 
	\end{Remark}
    \vspace{-0.8cm}
	\subsection{Verification of conservation properties}
	\vspace{-0.5cm}
	In this subsection, we will give and analyze the conservation properties of EMAC-TwoLevel(Stokes) scheme (\ref{OneLevel-EMAC-Nonlinear}),(\ref{TwoLevel-EMAC-Stokes-Step2}) and EMAC-TwoLevel(Newton) scheme (\ref{OneLevel-EMAC-Nonlinear}),(\ref{TwoLevel-EMAC-Newton-Step2}) on energy, momentum, and angular momentum. Before that, we briefly point out the key to balancing these three physical quantities related to the two two-level numerical schemes proposed herein. To this end, we need to  set the boundary conditions carefully. As pointed out in \cite{Linearized-EMAC-ANM-2019}, most physical boundary conditions and their numerical discretizations will alter the physical balances of quantities of interest; therefore, in order to solely investigate the impact of two-level linearized methods on the linear corrections separately separated from the contribution of the boundary conditions, the same approach in \cite{Linearized-EMAC-ANM-2019} is adopted here which states as assuming the finite element solution $(\boldsymbol{u}^{n+1}_d, p^{n+1}_d)$ and $\boldsymbol{f}^{n+1}$, $\forall n \geq 0$, is supported in a subset $\widehat{\Omega} \subsetneq \Omega$, i.e., there is a strip $S = \Omega \backslash \widehat{\Omega}$ along $\partial\Omega$ where each $\boldsymbol{u}^{n+1}_d$ is zeros.  \\
	\indent To analyze the two schemes' conservation laws, We write uniformly the two step-2s (\ref{TwoLevel-EMAC-Stokes-Step2}), (\ref{TwoLevel-EMAC-Newton-Step2}) as
	$$
	\setlength{\abovedisplayskip}{0.2cm}
	\begin{aligned}
	(\boldsymbol{u}^{n+1}_h - \boldsymbol{u}^{n}_h)/\Delta t &+ \nu (\nabla\boldsymbol{u}^{n+1}_h,\nabla\boldsymbol{v}_h) + NL_{emac}(\boldsymbol{u}^{n+1}_H,\boldsymbol{u}^{n+1}_h,\boldsymbol{v}_h) \\
	&- (\nabla\cdot \boldsymbol{v}_h,p^{n+1}_h) + (\nabla\cdot \boldsymbol{u}^{n+1}_h,q_h) = (\boldsymbol{f}^{n+1},\boldsymbol{v}_h),
    \end{aligned}
    \setlength{\belowdisplayskip}{0.2cm}
	$$
	We test $\boldsymbol{v}_h = \boldsymbol{u}^{n+1}_h, \widehat{\boldsymbol{e}_i}$ and $\widehat{\boldsymbol{\phi}_i} := \widehat{\boldsymbol{x} \times \boldsymbol{e}_i}$ respectively, and set $q_h=0$, $\boldsymbol{f}^{n+1}=\boldsymbol{0}$, $\nu=0$ when analyzing energy, where the tilde operator alters these quantities only in the strip $S$ along the boundary to ensure $\widehat{\boldsymbol{e}_i}, \widehat{\boldsymbol{\phi}_i} \in \boldsymbol{X}_h$. This gives
	$$
	\setlength{\abovedisplayskip}{0.2cm}
	\begin{aligned}
		\|\boldsymbol{u}^{n+1}_h\|^2_0 - \|\boldsymbol{u}^n_h\|^2_0 &= -\Delta tNL_{emac}(\boldsymbol{u}^{n+1}_H,\boldsymbol{u}^{n+1}_h,\boldsymbol{u}^{n+1}_h), \\
		\left(\boldsymbol{u}^{n+1}_h, \boldsymbol{e}_i\right) - \left(\boldsymbol{u}^n_h, \boldsymbol{e}_i\right) &= -\Delta tNL_{emac}(\boldsymbol{u}^{n+1}_H,\boldsymbol{u}^{n+1}_h,\boldsymbol{e}_i), \\
		\left(\boldsymbol{u}^{n+1}_h, \boldsymbol{\phi}_i\right) - \left(\boldsymbol{u}^n_h, \boldsymbol{\phi}_i\right) &= -\Delta tNL_{emac}(\boldsymbol{u}^{n+1}_H,\boldsymbol{u}^{n+1}_h,\boldsymbol{\phi}_i),
	\end{aligned}
	\setlength{\belowdisplayskip}{0.2cm}
	$$
	From this, we can conclude that the key to conserving the three quantities depends on whether or not the three trilinear terms above vanish in two two-level schemes. Specifically, we have the following two theorems concerning the two numerical schemes.
	\begin{Theorem}[Conservation properties for EMAC-TwoLevel(Stokes)]\label{EMAC-TwoLevel(Stokes)}
		The EMAC-TwoLevel(Stokes) scheme (\ref{OneLevel-EMAC-Nonlinear}),(\ref{TwoLevel-EMAC-Stokes-Step2}) conserves momentum and angular momentum in the absence of extra force $\boldsymbol{f}$; specifically, for $n=1,2,\cdots, N$, the solutions $\boldsymbol{u}^n_h$ of (\ref{OneLevel-EMAC-Nonlinear}),(\ref{TwoLevel-EMAC-Stokes-Step2}) satisfy
		$$
		\setlength{\abovedisplayskip}{0.2cm}
		\begin{aligned}
			\left(\boldsymbol{u}^n_h, \boldsymbol{e}_i\right) &= \left(\boldsymbol{u}^0_h, \boldsymbol{e}_i\right),  \\
			\left(\boldsymbol{u}^n_h, \boldsymbol{x} \times \boldsymbol{e}_i\right) &= \left(\boldsymbol{u}^0_h, \boldsymbol{x} \times \boldsymbol{e}_i\right),
		\end{aligned}
		\setlength{\belowdisplayskip}{0.2cm}
		$$
		and asymptotically conserves energy up to a spatial term $\mathcal{O}(H)$, after assuming $\Delta t \lesssim H^2$, in the sense of 
		$$
		\setlength{\abovedisplayskip}{0.2cm}
		\|\boldsymbol{u}^n_h\|^2_0 = \|\boldsymbol{u}^0_h\|^2_0 + \mathcal{O}(H), 
		\setlength{\belowdisplayskip}{0.2cm}
		$$
		where $\boldsymbol{e}_i(i=1,2,3)$ represents the $i$-th coordinate direction, $\boldsymbol{x}$ is spatial coordinate and $H$ denotes the length-scale of coarse mesh.  
	\end{Theorem}
	\begin{proof}
		\vspace{-0.5cm}
		As indicated above, we only need to analyze the trilinear term with three different test functions, and the trilinear term, at this point, is specific as
		$$
		\setlength{\abovedisplayskip}{0.2cm}
		NL_{emac}(\boldsymbol{u}^{n+1}_H,\boldsymbol{u}^{n+1}_h,\boldsymbol{v}_h) = b_{emac}(\boldsymbol{u}^{n+1}_H,\boldsymbol{u}^{n+1}_H,\boldsymbol{v}_h).
		\setlength{\belowdisplayskip}{0.2cm}
		$$
		Firstly, in order to prove momentum conservation, we use (\ref{PartInte-2-CONV}), (\ref{Rela-EMAC-CONV}), and the fact that $\nabla\cdot \boldsymbol{e}_i = 0$ and $\nabla \boldsymbol{e}_i = \boldsymbol{0}$, to obtain the vanishing trilinear term by simple calculation after setting $\boldsymbol{v}_h = \boldsymbol{e}_i$, which means the momentum is conservative. Secondly, to see angular momentum conservation, using (\ref{PartInte-2-CONV}), (\ref{Rela-EMAC-CONV}), the fact that $\nabla\cdot \boldsymbol{\phi}_i = \nabla\cdot(\boldsymbol{x} \times \boldsymbol{e}_i) = 0$, and direct calculation, we get $b_{emac}(\boldsymbol{u}^{n+1}_H,\boldsymbol{u}^{n+1}_H,\boldsymbol{\phi}_i) = 0$. This implies the conservation of angular momentum. Lastly, to see energy conservation asymptotically, we firstly use Theorem \ref{Theorem-OneLevel-EMAC-Conv} and Theorem \ref{Theorem-Conv-Newton} to estimate
        $$
        \setlength{\abovedisplayskip}{0.2cm}
        \begin{aligned}
        	\|\boldsymbol{u}^{n+1}_h-\boldsymbol{u}^{n+1}_H\|_1 
        	\lesssim \|\boldsymbol{u}^{n+1}-\boldsymbol{u}^{n+1}_h\|_1 + \|\boldsymbol{u}^{n+1}-\boldsymbol{u}^{n+1}_H\|_1 \lesssim (h + H^2 + \Delta t) + H^{-1}(H^2 + \Delta t).
        \end{aligned}
        \setlength{\belowdisplayskip}{0.2cm}
        $$
        Thus, if we assume $\Delta t \lesssim H^2$, we derive
        \begin{equation}\label{ConvRate-uh-uH}
        	\setlength{\abovedisplayskip}{0.2cm}
        	\|\boldsymbol{u}^{n+1}_h-\boldsymbol{u}^{n+1}_H\|_1 \lesssim H.
        	\setlength{\belowdisplayskip}{0.2cm}
        \end{equation}
        Therefore, we derive
         $$
         \setlength{\abovedisplayskip}{0.2cm}
         \begin{aligned}
         	b_{emac}(\boldsymbol{u}^{n+1}_H,\boldsymbol{u}^{n+1}_H,\boldsymbol{u}^{n+1}_h) = b_{emac}(\boldsymbol{u}^{n+1}_H,\boldsymbol{u}^{n+1}_H-\boldsymbol{u}^{n+1}_h,\boldsymbol{u}^{n+1}_h) \lesssim \|\boldsymbol{u}^{n+1}_H-\boldsymbol{u}^{n+1}_h\|_1\|\boldsymbol{u}^{n+1}_H\|_1\|\boldsymbol{u}^{n+1}_h\|_1 \lesssim H,
         \end{aligned}
         \setlength{\belowdisplayskip}{0.2cm}
         $$
         which means the trilinear term, when estimating energy, is $\mathcal{O}(H)$-small if we assume $\Delta t \lesssim H^2$:
         $$
         \setlength{\abovedisplayskip}{0.2cm}
         b_{emac}(\boldsymbol{u}^{n+1}_H,\boldsymbol{u}^{n+1}_H,\boldsymbol{u}^{n+1}_h) = \mathcal{O}(H).
         \setlength{\belowdisplayskip}{0.2cm}
         $$
         Thus, this proof is completed.    
    \end{proof}
    \vspace{-0.2cm}
    \begin{Remark}\label{Remark-HighOrder-Stokes}
    	A higher-order term about spatial length-scale $H$ in above the asymptotic conservation of energy may be derived through the sharper estimate of the trilinear term. For example, we firstly obtain stability of $\|\boldsymbol{u}^{n+1}_H\|_2$ directly by $\|\boldsymbol{u}^{n+1}_H\|_2 \leq \|\boldsymbol{u}^{n+1}\|_2 + \|\boldsymbol{u}^{n+1}-\boldsymbol{u}^{n+1}_H\|_2 \lesssim C+H^{-2}(H^2+\Delta t) \leq C$ if $\Delta t \lesssim H^2$, with which we can get $\|\boldsymbol{u}^{n+1}_h-\boldsymbol{u}^{n+1}_H\|_0 \lesssim \|\boldsymbol{u}^{n+1} - \boldsymbol{u}^{n+1}_h\|_0 + \|\boldsymbol{u}^{n+1} - \boldsymbol{u}^{n+1}_H\|_0 \lesssim \|\boldsymbol{u}^{n+1} - \boldsymbol{u}^{n+1}_h\|_1 + H^2+\Delta t \lesssim h + H^2+\Delta t$. Therefore, an alternative estimate about trilinear can be obtained as
    	$$
    	\setlength{\abovedisplayskip}{0.2cm}
    		b_{emac}(\boldsymbol{u}^{n+1}_H,\boldsymbol{u}^{n+1}_H,\boldsymbol{u}^{n+1}_h) 
    		\lesssim \|\boldsymbol{u}^{n+1}_H-\boldsymbol{u}^{n+1}_h\|_0\|\boldsymbol{u}^{n+1}_H\|_2\|\boldsymbol{u}^{n+1}_h\|_1 
    		\lesssim h + H^2+\Delta t.
        \setlength{\belowdisplayskip}{0.2cm}
    	$$
    	Therefore, this estimate could be utilized to obtain spatial higher order asymptotic conservation of energy.
    	\vspace{-0.5cm}
    \end{Remark}
	We then give another result about EMAC-TwoLevel(Newton) scheme (\ref{OneLevel-EMAC-Nonlinear}),(\ref{TwoLevel-EMAC-Newton-Step2}).
	\begin{Theorem}[Conservation properties for EMAC-TwoLevel(Newton)]
		The EMAC-TwoLevel(Newton) scheme (\ref{OneLevel-EMAC-Nonlinear}),(\ref{TwoLevel-EMAC-Newton-Step2}) conserves momentum and angular momentum in the absence of extra force $\boldsymbol{f}$ in the sense of, for $n=1,2,\cdots, N$, the solutions $\boldsymbol{u}^n_h$ of (\ref{OneLevel-EMAC-Nonlinear}),(\ref{TwoLevel-EMAC-Newton-Step2}) satisfies
		$$
		\setlength{\abovedisplayskip}{0.2cm}
		\begin{aligned}
			\left(\boldsymbol{u}^n_h, \boldsymbol{e}_i\right) &= \left(\boldsymbol{u}^0_h, \boldsymbol{e}_i\right),  \\
			\left(\boldsymbol{u}^n_h, \boldsymbol{x} \times \boldsymbol{e}_i\right) &= \left(\boldsymbol{u}^0_h, \boldsymbol{x} \times \boldsymbol{e}_i\right),
		\end{aligned}
	    \setlength{\belowdisplayskip}{0.2cm}
		$$
		and asymptotically conserves energy up to a spatial high-order term $\mathcal{O}(H^2)$ if assuming $\Delta t \lesssim H^2$:
		$$
		\setlength{\abovedisplayskip}{0.2cm}
		\|\boldsymbol{u}^n_h\|^2_0 = \|\boldsymbol{u}^0_h\|^2_0 + \mathcal{O}(H^2), 
		\setlength{\belowdisplayskip}{0.2cm}
		$$
		where $\boldsymbol{e}_i(i=1,2,3)$ represents the $i$-th coordinate direction, $\boldsymbol{x}$ is spatial coordinate and $H$ denotes the length-scale of coarse mesh. 
	\end{Theorem}
	\begin{proof}
		\vspace{-0.5cm}
		Similarly, we get the specific trilinear term as
		$$
		\setlength{\abovedisplayskip}{0.2cm}
		\begin{aligned}
			NL_{emac}(\boldsymbol{u}^{n+1}_H,\boldsymbol{u}^{n+1}_h,\boldsymbol{v}_h) 
			&= b_{emac}(\boldsymbol{u}^{n+1}_H,\boldsymbol{u}^{n+1}_h,\boldsymbol{v}_h) + b_{emac}(\boldsymbol{u}^{n+1}_h,\boldsymbol{u}^{n+1}_H,\boldsymbol{v}_h) - b_{emac}(\boldsymbol{u}^{n+1}_H,\boldsymbol{u}^{n+1}_H,\boldsymbol{v}_h)  \\
			&\triangleq NL^{Newton}_{emac}(\boldsymbol{u}^{n+1}_H,\boldsymbol{u}^{n+1}_h,\boldsymbol{v}_h).
		\end{aligned}
	    \setlength{\belowdisplayskip}{0.2cm}
		$$
		We use (\ref{PartInte-CONV}) twice to deduce that, for some constant $\boldsymbol{c}$, $b_{conv}(\boldsymbol{u},\boldsymbol{v},\boldsymbol{c}) = -((\nabla\cdot\boldsymbol{u})\boldsymbol{v},\boldsymbol{c})$ and $b_{conv}(\boldsymbol{c},\boldsymbol{v},\boldsymbol{w}) = -b_{conv}(\boldsymbol{c},\boldsymbol{w},\boldsymbol{v})$. Thus, using (\ref{Rela-2-EMAC-CONV}) and these two identities, we obtain 
		$$
		\setlength{\abovedisplayskip}{0.2cm}
		NL^{Newton}_{emac}(\boldsymbol{u}^{n+1}_H,\boldsymbol{u}^{n+1}_h,\boldsymbol{e}_i) = 0.
		\setlength{\belowdisplayskip}{0.2cm}
		$$
		To see angular momentum conservation, we firstly use $\nabla\cdot\boldsymbol{\phi}_i = \nabla\cdot(\boldsymbol{x}\times\boldsymbol{e}_i)=0$ to get $b_{conv}(\boldsymbol{\phi}_i,\boldsymbol{v},\boldsymbol{w}) = -b_{conv}(\boldsymbol{\phi}_i,\boldsymbol{w},\boldsymbol{v})$, therefore, by direct calculation we can deduce 
		$$
		\setlength{\abovedisplayskip}{0.2cm}
		NL^{Newton}_{emac}(\boldsymbol{u}^{n+1}_H,\boldsymbol{u}^{n+1}_h,\boldsymbol{\phi}_i) = 0.
		\setlength{\belowdisplayskip}{0.2cm}
		$$
		As for the energy conservation, we test with $\boldsymbol{v}_h=\boldsymbol{u}^{n+1}_h$ and using skew-symmetry property to get  
		$$
		\setlength{\abovedisplayskip}{0.2cm}
			NL^{Newton}_{emac}(\boldsymbol{u}^{n+1}_H,\boldsymbol{u}^{n+1}_h,\boldsymbol{\phi}_i) = 
			- b_{emac}(\boldsymbol{u}^{n+1}_h-\boldsymbol{u}^{n+1}_H,\boldsymbol{u}^{n+1}_h-\boldsymbol{u}^{n+1}_H,\boldsymbol{u}^{n+1}_h). 
	    \setlength{\belowdisplayskip}{0.2cm}
		$$
		Thus, assuming $\Delta t \lesssim H^2$ and using (\ref{ConvRate-uh-uH}), we can bound the trilinear term above as 
		$$
		\setlength{\abovedisplayskip}{0.2cm}
		|b_{emac}(\boldsymbol{u}^{n+1}_h-\boldsymbol{u}^{n+1}_H,\boldsymbol{u}^{n+1}_h-\boldsymbol{u}^{n+1}_H,\boldsymbol{v}_h)|
		\lesssim \|\boldsymbol{u}^{n+1}_h-\boldsymbol{u}^{n+1}_H\|_1^2\|\boldsymbol{u}^{n+1}_h\|_1
		\lesssim H^2,
		\setlength{\belowdisplayskip}{0.2cm}
		$$
		which means the trilinear term, when estimating energy, is $\mathcal{O}(H^2)$-small if we assume $\Delta t \lesssim H^2$:
		$$
		\setlength{\abovedisplayskip}{0.2cm}
		b_{emac}(\boldsymbol{u}^{n+1}_h-\boldsymbol{u}^{n+1}_H,\boldsymbol{u}^{n+1}_h-\boldsymbol{u}^{n+1}_H,\boldsymbol{v}_h) = \mathcal{O}(H^2).
		\setlength{\belowdisplayskip}{0.2cm}
		$$
		Thus, this completes the proof. 
	\vspace{-0.5cm}
    \end{proof}
    \vspace{-0.5cm}
	\section{Numerical experiments}\label{section-4}
	\vspace{-0.5cm}
	In this section, three numerical experiments will be presented. We adopt Taylor-Hood elements and BDF2 time stepping in these three experiments, and we remark that this BDF2 EMAC-TwoLevel(Newton/Stokes) is used for the sake of balancing between computational efficiency and optimal numerical accuracy and would not damage the conclusions about the (asymptotic) conservative properties. Moreover, we point out that only the results of EMAC-TwoLevel(Newton) are presented, and similar results about EMAC-TwoLevel(Stokes) are also obtained but are omitted here for brevity.
	\vspace{-0.5cm}
	\subsection{Convergence rate verification}
	\vspace{-0.5cm}
	The first test is taken from \cite{MLG-2015-SISC} on $\Omega = [0,1]^2$, and the exact solution takes the form:
	$$
	\setlength{\abovedisplayskip}{0.2cm}
	\begin{aligned}
		\boldsymbol{u}(\boldsymbol{x}, t) &= [\sin ^2 \pi x \sin 2 \pi y, -\sin ^2 \pi y \sin 2 \pi x]^{\text{T}}\cdot(1+\sin \pi t), \\
		p(\boldsymbol{x}, t)   &= (1+\sin \pi t) \cos \pi x \cos \pi y,
	\end{aligned}
    \setlength{\belowdisplayskip}{0.2cm}
	$$
    with $\nu=1$, $T=1$, and $\boldsymbol{f}$ is obtained by direct calculation. We take $\Delta t = h^{3/2}$ and define two types of pressure approximation errors and convergence rates. The first one is $\|p^n_{primal} - p^n_{h,primal}\|_{L^2}$ that is raised to relate to primal pressure variable $p_{primal}(t):=p(t)$ in continuous NSEs (\ref{ContEqua}) and $p^n_{h,primal}$ is recovered \emph{a posteriori} from discrete EMAC pressure $p^n_{h}$ in (\ref{TwoLevel-EMAC-Stokes-Step2}) or (\ref{TwoLevel-EMAC-Newton-Step2}) by $p^n_{h,primal} := p^n_{h} + 1/2\cdot|\boldsymbol{u}^n_h|^2$; the second is $\|p^n_{emac} - p^n_{h,emac}\|_{L^2}$, where $p_{emac}(\boldsymbol{x},t) := p(\boldsymbol{x},t) - 1/2\cdot|\boldsymbol{u}^n(\boldsymbol{x},t)|^2 + \lambda$ is defined \emph{a priori} and $\lambda := \int_{\Omega}1/2\cdot|\boldsymbol{u}(\cdot,t)|^2\text{d}x$, this form is relevant to the numerical EMAC pressure $p^n_{h,emac} := p^n_{h}$ obtained in (\ref{TwoLevel-EMAC-Stokes-Step2}) or (\ref{TwoLevel-EMAC-Newton-Step2}). The errors and optimal rates about velocity and two types of pressure from  \autoref{FEVeloConvOrder} confirm our theoretical estimates in Theorem \autoref{Theorem-Conv-Stokes} and Theorem \ref{Theorem-Conv-Newton}.
    \vspace{-0.2cm}
	\begin{table}
		\centering
		\fontsize{10}{10}
		\begin{threeparttable}
			\caption{Errors and convergence orders for EMAC-TwoLevel(Newton) scheme with $\Delta t = h^{3/2}$, $\nu=1$.}\label{FEVeloConvOrder}
			\begin{tabular}{c|c|c|c|c|c|c|c|c|c}
				\toprule
				\multirow{2.5}{*}{1/h} & \multirow{2.5}{*}{1/H} & \multicolumn{2}{c}{$\left\|\boldsymbol{u}^N - \boldsymbol{u}^N_h\right\|_{L^2}$} & \multicolumn{2}{c}{$\left\|\boldsymbol{u}^N - \boldsymbol{u}^N_h\right\|_{H^1}$} & \multicolumn{2}{c}{$\left\|p^N_{primal} - p^N_{h,primal}\right\|_{L^2}$} & \multicolumn{2}{c}{$\left\|p^N_{emac} - p^N_{h,emac}\right\|_{L^2}$} \cr
				\cmidrule(lr){3-4} \cmidrule(lr){5-6} \cmidrule(lr){7-8}  \cmidrule(lr){9-10}
				&  & error & rate & error & rate & error & rate & error & rate \cr
				\midrule
				4  & 2 &  1.6283e$-$02  &  -     &  5.0917e$-$01   & -     &  1.2137e$-$01   & -    &  1.2740e$-$01   & -    \cr
				16 & 4 &  2.4919e$-$04  &  3.02  &  3.5767e$-$02   & 1.92  &  7.5670e$-$03   & 2.00 &  2.5078e$-$03   & 2.83 \cr
				36 & 6 &  2.2059e$-$05  &  2.99  &  7.1190e$-$03   & 1.99  &  1.5602e$-$03   & 1.95 &  3.8980e$-$04   & 2.30 \cr
				64 & 8 &  3.9393e$-$06  &  2.99  &  2.2555e$-$03   & 2.00  &  4.9727e$-$04   & 1.99 &  1.1803e$-$04   & 2.08 \cr
				\bottomrule
			\end{tabular}
		\end{threeparttable}
	\vspace{-0.5cm}
	\end{table}
    \vspace{-0.2cm}
	\subsection{Lattice vortex problem}
	\vspace{-0.5cm}
	The lattice vortex problem will be considered as the second test problem, and it was also commonly used in EMAC-related research(see, e.g., \cite{Pres-Re-Robust-SV-UNSE-2017-JNM,Linearized-EMAC-ANM-2019}) since this system holds the conserved zeros momentum and angular momentum all the time. This is a challenging problem because there are several spinning vortices whose edges touch, which renders it difficult to be numerically resolved. The true solution takes the form of
	$$
	\setlength{\abovedisplayskip}{0.2cm}
	\begin{aligned}
		\boldsymbol{u} &= [\sin (2 \pi x) \sin (2 \pi y), \cos (2 \pi x) \cos (2 \pi y)]^{\text{T}}\cdot\exp ^{-8 \nu \pi^2 t} , \\
		p &= -\frac{1}{2}\left(\sin ^2(2 \pi x)+\cos ^2(2 \pi y)\right)  \exp ^{-16 \pi^2 t},
	\end{aligned}
    \setlength{\belowdisplayskip}{0.2cm}
	$$
	with $\boldsymbol{f}=\boldsymbol{0}$ and $\nu = 10^{-7}$, $\boldsymbol{u}_0 := \boldsymbol{u}(\cdot,0)$. This test is performed on $\Omega = [0,1]^2$ with time up to $T=5$. \\
	\indent We simulate this problem by two-level scheme with Newton correction, and to show this scheme is indeed preserving energy, momentum, and angular momentum obviously, we also supply the simulations about other four common-used formulations linearized by the same two-level methods so as to give a contrast to EMAC form. Apart form the forms of EMAC and Conv defined before, the other three(Skew, Rota, and Dive ) formulations are defined respectively as
	 $$
	 \setlength{\abovedisplayskip}{0.2cm}
	 \begin{aligned}
	 	b_{\text{skew}}(\boldsymbol{u},\boldsymbol{v},\boldsymbol{w}) &:= \frac12((\boldsymbol{u}\cdot \nabla)\boldsymbol{v},\boldsymbol{w}) - \frac12((\boldsymbol{u}\cdot \nabla)\boldsymbol{w},\boldsymbol{v}),  \\
	 	b_{\text{rota}}(\boldsymbol{u},\boldsymbol{v},\boldsymbol{w}) &:= ((\nabla \times \boldsymbol{u})\times \boldsymbol{v},\boldsymbol{w}),  \\
	 	b_{\text{dive}}(\boldsymbol{u},\boldsymbol{v},\boldsymbol{w}) &:= (\nabla \cdot (\boldsymbol{u} \otimes \boldsymbol{v}),\boldsymbol{w}).
	 \end{aligned}
     \setlength{\belowdisplayskip}{0.2cm}
	 $$
    The coarse and fine mesh, with $H=1/18,h=1/36$, are obtained by successively refined partitioning of $\Omega$ with SWNE diagonals and time-step $\Delta t = 0.01$. \autoref{Figure-FiveConv} shows results of energy, momentum, angular momentum, and $L^2$ velocity error versus time for each of the different formulations. We see that, except for EMAC, the other four forms all become unstable and blow up.
    \vspace{-0.5cm}
    \begin{figure}
    	\centering
    	\includegraphics[width=1\linewidth]{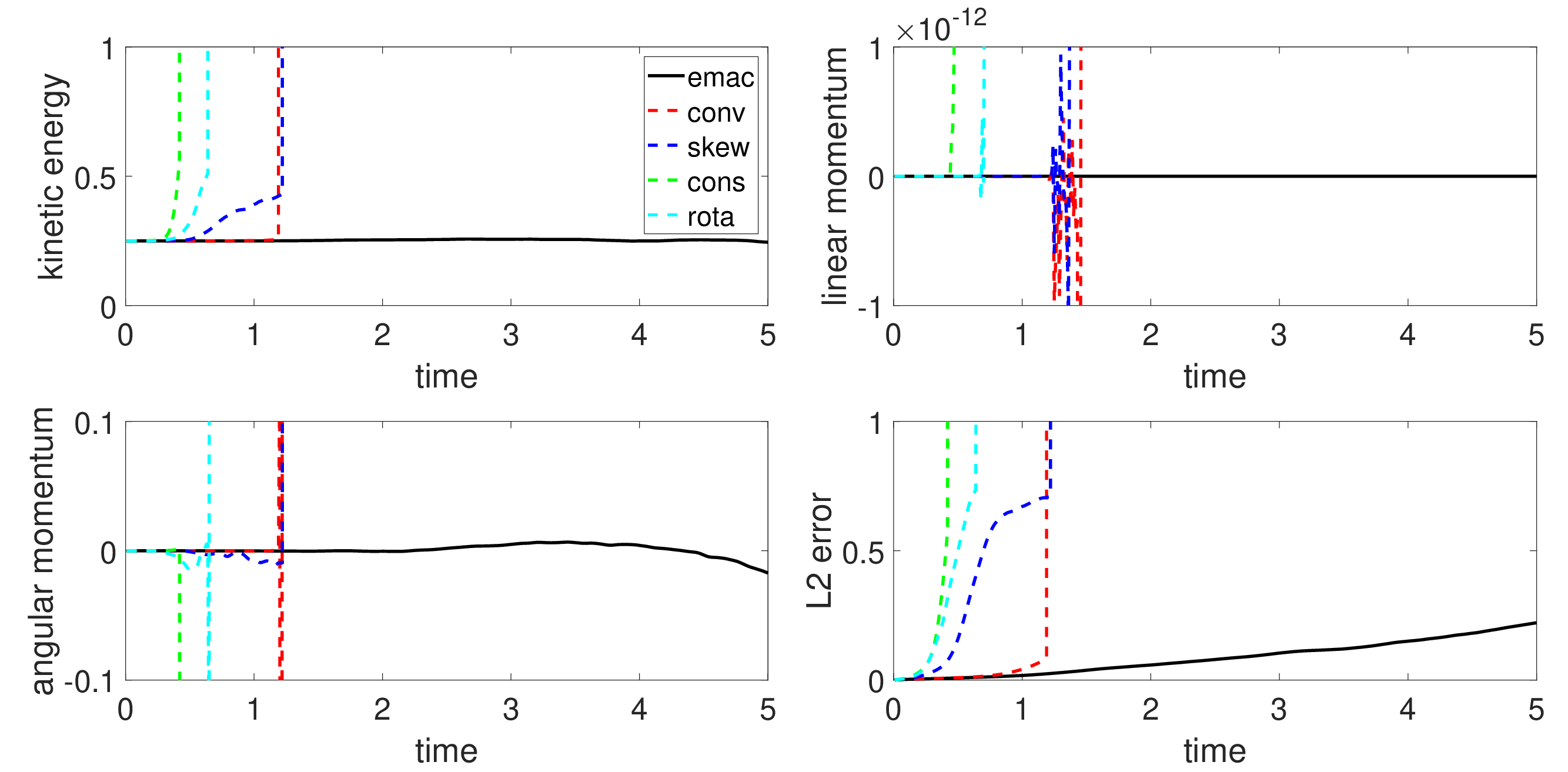}
    	\caption{Shown above are plots of time versus energy, momentum, angular momentum, and $L^2$ velocity error, for EMAC, Conv, Skew, Rota, and Dive formulations, for lattice vortex problem.}\label{Figure-FiveConv}
    \end{figure}
	\subsection{Flow past a cylinder in 2D}
	\vspace{-0.5cm}
	Here, we present a benchmark test to validate our numerical schemes. The simulation is a 2D channel flow past a cylinder on a channel with a hole $\Omega=\{(0,2.2) \times(0,0.41)\} \backslash\left\{\boldsymbol{x}:(\boldsymbol{x}-(0.2,0.2))^2 \leq 0.05^2\right\}$. There is no external forcing, i.e., $\boldsymbol{f}=\boldsymbol{0}$, and $\nu = 10^{-3}$. The parabolic inflow and outflow profile
	$$
	\setlength{\abovedisplayskip}{0.2cm}
	\boldsymbol{u}(0, y)=\boldsymbol{u}(2.2, y)=0.41^{-2}(1.2 y(0.41-y), 0), \quad 0 \leq y \leq 0.41,
	\setlength{\belowdisplayskip}{0.2cm}
	$$
	is prescribed, no-slip conditions are imposed at the other boundaries.  \\
	\indent The benchmark parameters are the drag and lift coefficients at the cylinder, and in order to reduce the negative effect of line integral on the boundary of the cylinder, we utilize the mean values of the inflow velocity $\bar{U}=1$ and the diameter of the cylinder $D=0.1$ to adopt the volume integral defined in \cite{benchmark-John-2001-IJNMF} as 
	$$
	\setlength{\abovedisplayskip}{0.2cm}
	\begin{aligned}
		c_d(t) & =-\frac{2}{D \bar{U}^2}\left[\nu\left(\nabla \boldsymbol{u}(t), \nabla \boldsymbol{v}_d\right)+b_{EMAC}\left(\boldsymbol{u}(t), \boldsymbol{u}(t), \boldsymbol{v}_d\right)-\left(p(t), \nabla \cdot \boldsymbol{v}_d\right)\right], 
	\end{aligned}
\setlength{\belowdisplayskip}{0.2cm}
$$
		$$
		\setlength{\abovedisplayskip}{0.2cm}
		\begin{aligned}
		c_l(t) & =-\frac{2}{D \bar{U}^2}\left[\nu\left(\nabla \boldsymbol{u}(t), \nabla \boldsymbol{v}_l\right)+b_{EMAC}\left(\boldsymbol{u}(t), \boldsymbol{u}(t), \boldsymbol{v}_l\right)-\left(p(t), \nabla \cdot \boldsymbol{v}_l\right)\right],
	\end{aligned}
	\setlength{\belowdisplayskip}{0.2cm}
	$$
	for arbitrary functions $\boldsymbol{v}_d \in \boldsymbol{H}^1$(resp. $\boldsymbol{v}_l \in \boldsymbol{H}^1$) such that $\boldsymbol{v}_d = (1, 0)^{\text{T}}$(resp. $\boldsymbol{v}_l = (0, 1)^{\text{T}}$) on the boundary of the cylinder and vanishes on the other boundaries. Another relevant quantity of interest is the pressure difference between the front and back of the cylinder 
	$
	\setlength{\abovedisplayskip}{0.2cm}
	\Delta p(t) = p(0.15,0.2,t) - p(0.25,0.2,t).
	\setlength{\belowdisplayskip}{0.2cm}
	$
	The following \autoref{FlowPast} shows the development of those three quantities which demonstrate the effectiveness of our proposed schemes after being compared with the reference values in \cite{benchmark-John-2004-IJNMF}.
    \begin{figure}[htbp!]
    	\centering
    	\includegraphics[width=1\linewidth]{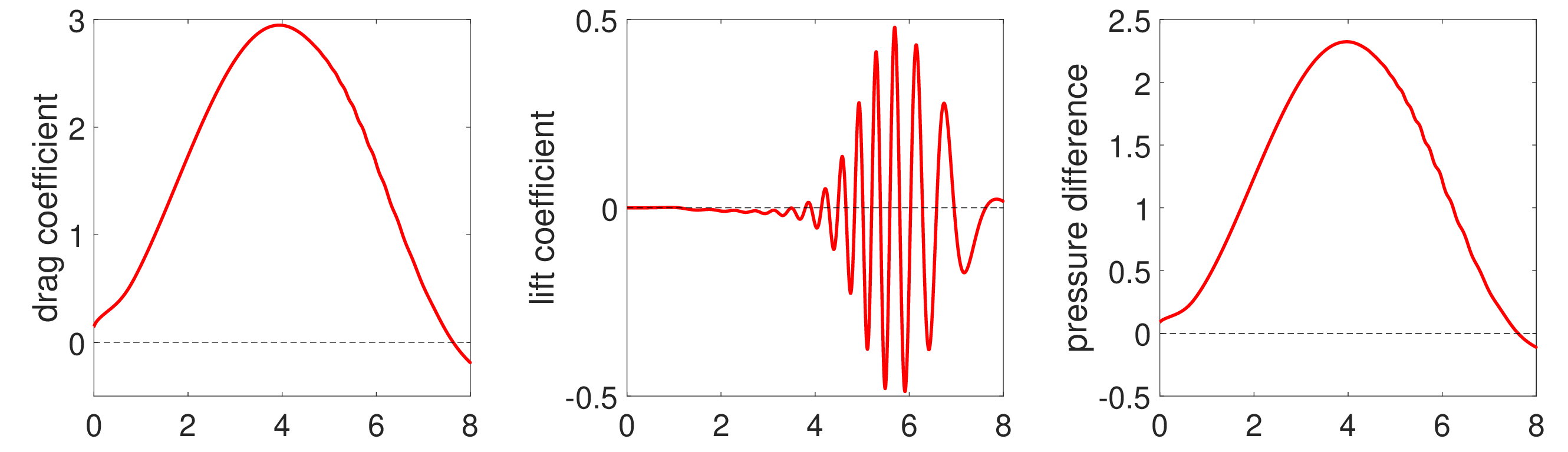}
    	\caption{Evolution of $c_d(t)$, $c_l(t)$ and $\Delta p(t)$ in the simulation of flow past a cylinder with EMAC-TwoLevel(Newton) scheme.}\label{FlowPast}
    \end{figure}
    \vspace{-0.8cm}
    \section{Conclusions}
    \vspace{-0.5cm}
	We have studied the conservation properties of NSEs, written in EMAC formulation, after being linearized by two two-level methods. In particular, we have proven that the two-level methods, discretized from EMAC NSEs and equipped with appropriate correction steps, could preserve momentum, angular momentum, and asymptotically preserve energy. Theoretical results have been provided, as well as three numerical experiments. The first one has validated the optimal convergence rates, which especially includes two newly proposed pressure approximations in EMAC NSEs. The second test has shown that the two-level methods, equipped with appropriate correction, can indeed preserve energy, momentum, and angular momentum. Finally, the last benchmark test has confirmed the effectiveness of our new numerical schemes. In the future, we plan to complement the EMAC pressure's theoretical error estimates and to extend this methodology to other conservative and efficient linearized methods for NSEs.
	\vspace{-0.8cm}
	\section*{Data availability}
	\vspace{-0.8cm}
	Data will be made available on reasonable request.
	\vspace{-0.8cm}
	\section*{Declaration of competing interest}
	\vspace{-0.8cm}
	The authors declare that they have no known competing financial interests or personal relationships.
	\vspace{-0.8cm}
	\section*{CRediT authorship contribution statement}
	\vspace{-0.8cm}
	\textbf{Xi Li:} Conceptualization, Methodology, Software, Validation, Writing – original draft. \textbf{Minfu Feng: }Conceptualization, Funding acquisition, Methodology, Writing – review \& editing. 
    \vspace{-0.5cm}
	

	\bibliographystyle{abbrv}
	\bibliography{Reference}	
\end{document}